\documentclass{amsart}
\usepackage{latexsym}
\usepackage{amsmath}
\usepackage{amssymb}
\usepackage[all]{xy}

\newtheorem{thm}{Theorem}[section]
\newtheorem{prop}[thm]{Proposition}

\newtheorem{cor}[thm]{Corollary}
\newtheorem{lem}[thm]{Lemma}
\theoremstyle{definition}
\newtheorem{defn}[thm]{Definition}

\newtheorem{assertion}[thm]{Assertion}

\theoremstyle{remark}
\newtheorem{rem}[thm]{Remark}
\newtheorem{ex}[thm]{Example}

\newcommand{\K}{R}
\newcommand{\G}{{\mathcal G}}
\newcommand{\C}{{\mathcal C}}
\newcommand{\T}{{\mathcal T}}
\newcommand{\calK}{{\mathcal K}}
\newcommand{\e}{\varepsilon}
\newcommand{\D}{\mathcal{D}}

\newcommand{\mapright}[1]{%
\smash{\mathop{%
 \hbox to 1cm{\rightarrowfill}}\limits_{#1}}}
\newcommand{\maprightd}[2]{%
\smash{\mathop{%
 \hbox to 1.2cm{\rightarrowfill}}\limits^{#1}\limits_{#2}}}
\newcommand{\mapleft}[1]{%
\smash{\mathop{%
 \hbox to 1cm{\leftarrowfill}}\limits_{#1}}}
\newcommand{\mapleftu}[1]{%
\smash{\mathop{%
 \hbox to 0.8cm{\leftarrowfill}}\limits^{#1}}}
\newcommand{\maprightu}[1]{%
\smash{\mathop{%
 \hbox to 1cm{\rightarrowfill}}\limits^{#1}}}
\newcommand{\maprightud}[2]{%
\smash{\mathop{%
 \hbox to 1cm{\rightarrowfill}}\limits^{#1}_{#2}}}
\newcommand{\mapleftud}[2]{%
\smash{\mathop{%
 \hbox to 1cm{\leftarrowfill}}\limits^{#1}_{#2}}}

\newcounter{eqn}[section]

\def\theeqn{\textnormal{(\thesection.\arabic{eqn})}}

\def\eqnlabel#1{%
 \refstepcounter{eqn}%
 \label{#1}%
 \leqno{\theeqn}}

\begin{document}

\title{
On Mitchell's embedding theorem for a quasi-schemoid  \\
}

\footnote[0]{



Department of Mathematical Sciences, 
Faculty of Science,  
Shinshu University,   
Matsumoto, Nagano 390-8621, Japan   
e-mail:{\tt kuri@math.shinshu-u.ac.jp} 

e-mail:{\tt momose@math.shinshu-u.ac.jp}
}

\author{Katsuhiko KURIBAYASHI and Yasuhiro MOMOSE}

\maketitle

\begin{abstract}
A quasi-schemoid is a small category whose morphisms are colored with appropriate combinatorial data. 
In this paper, Mitchell's embedding theorem for a {\it tame} schemoid is established. The result allows us to give 
a cofibrantly generated model category structure to the category of chain complexes over a functor category with 
a schemoid as the domain. 
Moreover, a notion of Morita equivalence for schemoids is introduced and discussed. 
In particular, we show that every Hamming scheme of binary codes is Morita equivalent to 
the association scheme arising from the cyclic group of order two. 
In an appendix, 
we construct a new schemoid from an abstract simplicial complex, whose Bose-Mesner algebra is closely related 
to the Stanley-Reisner ring of the given complex.  
\end{abstract}

\section{Introduction} There are two crucial categories for 
representation theory of small categories including groups and quivers. 
One is a module category and another one is a functor category. 
Mitchell's embedding theorem \cite[Theorem 7.1]{Mitchell} 
states that these categories are equivalent provided the small category, which we deal with, has finite many objects.  

Association schemes, ASs for short, are significant subjects in algebraic combinatorics; see \cite{B-I, D, Z_1995}.  
These subjects give rise to the so-called Bose-Mesner algebras (adjacency algebras) and 
the study of the algebras creates applications in the theory of codes and designs; see for example \cite{P-Z}. 
An important point is that the category of finite groups is embedded in the category of ASs
in the sense of Hanaki; see \cite{Z_book, H, F_1}. If an association scheme is thin (in the sense of \cite{Z_book}), 
then its Bose-Mesner algebra is just the group ring of the corresponding group. Thus, representation theory of ASs 
is developed in the module categories of their Bose-Mesner algebras. 
However, until today there is few study on ASs dealing with their categorical and homological structures such as group cohomology. 

Very recently, Matsuo and the first author \cite{K-M} have introduced the notion of {\it quasi-schemoids} 
generalizing that of ASs from a small categorical point of view.  Roughly speaking, 
a quasi-schemoid is a small category whose morphisms are colored with appropriate combinatorial data. 
In \cite{Momose}, Momose has considered 
representation theory for quasi-schemoids with module categories of their Bose-Mesner algebras. 
In this manuscript, we develop {\it another} representation theory, namely that based on an appropriate 
functor category with a quasi-schemoid as the domain.  
It is worthwhile to remark that the two categories for representation theory 
of schemoids are not equivalent in general even the set of objects 
in the underlying category of a given quasi-schemoid is finite.  That is, Mitchell's embedding theorem does not necessarily hold 
in our context; see Proposition \ref{prop:group_Hamming} and Remark \ref{rem:example_equivalences}. 

One of the aims of this manuscript is to give a class of  schemoids in which Mitchell's embedding theorem holds; 
see Theorem \ref{thm:Mitchell's_thm1}. Such schemoids are called {\it tame}. 
Our functor category for a schemoid is a subcategory, but not full, of the usual one 
for the underlying category. Therefore, the existence of left and right adjoints to a restriction functor is not immediate.  
We will also discuss this problem; see Theorem \ref{thm:adjoints}.  

An outline for the article is as follows. In Section 2, 
we describe our main theorems concerning Mitchell's embedding and   
adjoint functors on functor categories of schemoids. By employing the adjoint pair, we define 
a cofibrantly generated model category structure on the category of chain complexes over a functor category 
with a schemoid as the domain; see Theorem \ref{thm:model_category_str}. 
Moreover, schemoid cohomology of a morphism between 
schemoids and a notion of Morita equivalence of schemoids are introduced. 
In Section 3, after defining a tame schemoid explicitly, 
we prove our main theorems. 
Section 4 concerns examples of schemoid cohomology and a Morita equivalence. 
In particular, we shall show that every Hamming scheme of binary codes is Morita equivalent to the association scheme 
arising from the cyclic group of order two; see Proposition \ref{prop:group_Hamming}. 
Section 5 explores an invariant for Morita equivalence which is induced by a functor between underlying categories. 

In Appendix 1, we construct a new schemoid from an abstract simplicial complex, 
whose schemoid cohomology is investigated in Section 4.  
This subject is very interesting in its own right. In fact, we show that its Bose-Mesner algebra is closely related 
to the Stanley-Reisner ring of the given complex. In consequence, such algebras give a complete invariant 
for isomorphism classes of finite simplicial complexes; see Assertion \ref{assertion:simplicial_complexes}.  
Moreover, the category of open sets of a topological space, whose morphisms are inclusions, admits a schemoid structure. In 
consequence, we will see that a functor category of the schemoid is an abelian subcategory of the category of presheaves over the 
given space. 

\section{Main theorems}

In what follows, a quasi-schemoid \cite{K-M} is referred to as a schemoid. 
We begin by recalling the definition of a schemoid. 
Let $\C$ be a small category 
and  $S$ a partition of the set $mor(\C)$ of all morphisms in $\C$; that is, $mor(\C) =\coprod_{\sigma \in S}\sigma$.
The pair $(\C, S)$ is called a {\it schemoid} 
if the set $S$ satisfies the condition that for a triple $\sigma, \tau, \mu \in S$ 
and for any morphisms $f$, $g$ in $\mu$, as a set 
$$
(\pi_{\sigma\tau}^\mu)^{-1}(f) \cong (\pi_{\sigma\tau}^\mu)^{-1}(g), 
$$ 
where $\pi_{\sigma\tau}^\mu : \pi_{\sigma\tau}^{-1}(\mu) \to \mu$ denotes 
the restriction of the concatenation map 
$$\pi_{\sigma\tau} : \sigma \times_{ob(\C)}\tau:=\{(f, g) \in \sigma \times \tau \mid s(f) = t(g)\} \to mor(\C).$$
We call the cardinality $p_{\sigma\tau}^\mu$  of the set $(\pi_{\sigma\tau}^\mu)^{-1}(f)$ a {\it structure constant}. 
Thus it seems that a schemoid is a category all whose morphisms are colored according to the condition above 
on a partition of the set of morphisms. 
As is seen below,  such a condition plays an important role in constructing an algebra with a schemoid.    

Let $\C$ be a category with $mor (\C)$ finite and $R$ a commutative ring with unit.   
The underlying category defines an $\K$-free module $\K\C$ generated by all morphisms of $\C$. 
For generators $f$ and $g$, define the product of $f$ and $g$ by 
$$
gf = \begin{cases}
g \circ f  & \text{if $g$ and $f$ are composable}  \\
0 & \text{otherwise}. 
\end{cases}
$$
Then we have a $\K$-algebra $\K\C$ which is called the {\it category algebra} of $\C$.
Let $(\C, S)$ be a schemoid with $mor (\C)$ finite. 
For any $\sigma$ and $\tau$ in $S$, an equality 
$$
(\sum_{s\in \sigma} s) \cdot (\sum_{t\in \tau} t) = \sum_{\mu\in S} p_{\sigma\tau}^\mu (\sum_{u\in \mu} u)
$$ 
holds in the category algebra $\K \C$ of $\C$. Thus one has  
a subalgebra 
$
\K(\C, S)
$
of $\K\C$ generated by the elements $(\sum_{s\in \sigma} s)$ for all $\sigma \in S$. 
The subalgebra is referred to 
as the {\it Bose-Mesner algebra} of the schemoid $(\C, S)$. 

\begin{ex}\label{ex:1}
For a small category $\C$, we define a partition $S$ by $S=\{\{f\}\}_{f \in mor(\C)}$; that is, all morphisms have 
pairwise different colors. 
Then we see that $\calK(\C):=(\C, S)$ is a schemoid. 
It is called a {\it discrete schemoid}.  
Observe that the Bose-Mesner algebra is the category algebra of the underlying category $\C$. 
\end{ex}

We recall the definition of an association scheme. 
Let $X$ be a finite set and $S$ a partition of $X\times X$, namely a subset of the power set $2^{X\times X}$. 
Assume that the subset $1_X:=\{ (x, x) \mid x \in X\}$ and $g^*:=\{(y, x) \mid (x, y) \in g\}$ for each 
$g \in S$ are in $S$.  Then the pair $(X, S)$ is called an 
{\it association scheme} if for all 
$e, f, g \in S$, there exists an integer $p_{ef}^g$ such that for any $(x, z) \in g$ 
$$
p_{ef}^g=\sharp \{y \in X \mid (x, y)\in e \ \text{and} \ (y, z) \in f \}. 
$$
Observe that $p_{ef}^g$ is independent of the choice of $(x, z) \in g$. 

\begin{ex}\label{ex:2}
For an association scheme $(X, S)$, we define a quasi-schemoid $\jmath(X, S)$ by the pair $(\C, V)$ for which 
$ob(\C)=X$, $\text{Hom}_\C(y, x) =\{(x, y)\} \subset X\times X$ and $V =S$, where the composite of morphisms 
$(z, x)$ and $(x, y)$ is defined by $(z, x) \circ (x, y) = (z, y)$.  It follows that the Bose-Mesner algebra is nothing but the original 
adjacency algebra of the association scheme; see \cite[Example 2.6 (i)]{K-M}.
\end{ex}

We refer the reader to \cite[Section 2]{K-M} for more examples of schemoids.  

Let $(\C, S)$ and $(\D, S')$ be schemoids. Then a functor $u : \C \to \D$ between underlying categories is called a 
{\it morphism of schemoids}, denoted $u :  (\C, S) \to (\D, S')$, if for $\sigma \in S$, there exists an element $\tau \in S'$ 
such that $u(\sigma) \subset \tau$. Observe that such an element $\tau$ is determined uniquely because $S'$ is a partition of 
$mor(\D)$. We denote by $q\mathsf{ASmd}$ the category of schemoids. 

Let $R\text{-Mod}$ be the category of modules over a commutative ring $\K$ with unit. 
Though the module category is not small, we regard it as a discrete schemoid, denoted $\T$, 
whose morphisms have distinct colors. 
For morphisms $f$ and $g$ in a schemoid $(\C, S)$, 
we say that  $f$ is {\it equivalent} to $g$, denoted $f \sim_S g$, if $f$ and $g$ are contained in a common set 
$\sigma \in S$. 
For morphisms  $u, v : (\C,S) \to \T$ of schemoids, a natural transformation 
$\eta : u   \Rightarrow  v$ is called {\it locally constant} if $\eta_x = \eta_y$ whenever $id_x \sim_S id_y$.  

We define $\T^{(\C,S)}$ to be a category whose objects are morphisms of schemoids from $(\C,S)$ to $\T$ and 
whose morphisms are locally constant natural transformations. Observe that $\T^{(\C,S)}$ is an abelian 
subcategory of the functor category $\T^{\C}$, but not full in general. 

In the category $\mathsf{Cat}$ of small categories, natural transformations give a notion of homotopy between functors. 
By employing the notion, a homotopy relation in the category $q\mathsf{ASmd}$ is defined in \cite{K}. We here recall it. 
Let $I$ be the discrete schemoid with objects $0$ and $1$ whose only non-trivial morphism is an arrow 
$u : 0 \to 1$.

\begin{defn}\label{defn:Homotopy}  
Let $F, G : (\C, S) \to (\D, S')$ be morphisms between the schemoids $(\C, S)$ and  $(\D, S')$ in $q\mathsf{ASmd}$. 
We write $H: F \Rightarrow G$ if $H$ is a morphism from  $(\C, S) \times I$ to $(\D, S')$ in  $q\mathsf{ASmd}$ with  
$H\circ \e_0 = F$ and $H\circ \e_1 = G$. Here $(\C, S) \times I$ denotes the product of the quasi-schemoids 
and 
$\e_i : (\C, S) \to (\C, S) \times I$ is the morphism of quasi-schemoids defined by 
$\e_i(a) = (a, i)$ for an object $a$ in $\C$ and $\e_i(f)= (f,  1_i)$ for a morphism $f$ in $\C$.  
We call the morphism $H$ above a {\it homotopy} from $F$ to $G$. 
A morphism $F$ is {\it equivalent} to $G$, 
denoted $F\sim G$, if there exists a homotopy from $F$ to $G$ or that from $G$ to $F$. 
\end{defn}

If there exists a homotopy $H : (\C, S) \times I \to \T$ from functors 
$F$ to $G$,  then we have a commutative diagram 
$$
\xymatrix@C35pt@R20pt{
{H(x, 0)}  \ar@{->}[r]^{H(id_{x}, u)} \ar[dr]^{H(\varphi,u)}   \ar@{->}[d]_{F(\varphi)=H(\varphi, 1_0)}  &  {H(x, 1)} 
\ar@{->}[d]^{H(\varphi, 1_1)=G(\varphi)}\\
{H(y, 0)} \ar@{->}[r]_{H(id_{y}, u)}  &  {H(y, 1)}
}
$$
in the underlying category $\T$  for a morphism $\varphi : x \to y$. Suppose that $id_x \sim_S id_y$, then 
$H(id_x, u) = H(id_y, u)$ because the homotopy $H$ preserves the partition $S$. 
Thus the definition of the morphism in the functor category $\T^{(\C,S)}$ is natural in our context.

We will introduce a class of schemoids which are called {\it tame} in the next section. 
A discrete schemoid and a schemoid associated with a groupoid are tame; 
see Proposition \ref{prop:groupoids}. 
Mitchell's embedding theorem for tame schemoids is established.

\begin{thm} {\em (See Theorem \ref{thm:Mitchell's_thm} for a more precise version.)} \label{thm:Mitchell's_thm1} 
Let $(\C, S)$ be a tame schemoid. Then 
the functor category $\T^{(\C, S)}$
is equivalent to the module category $\K(\C, S)\text{\em -Mod}$ if the set $\{id_x\}_{x \in ob \C}/\sim_S$ is finite and 
every structure constant is less than or equal to $1$. 
\end{thm}

In general, the functor category $\T^{(\C, S)}$ is {\it not} equivalent to the module category $\K(\C, S)\text{-Mod}$;
see Remark \ref{rem:example_equivalences} (i). In this manuscript, we mainly focus our attention on the functor categories of 
schemoids. Here a notion of Morita equivalence for schemoids is proposed. 

\begin{defn}
Schemoids $(\C, S_{\C})$ and $(\C', S_{\C'})$ are {\it Morita equivalent} if the functor categories 
$\T^{(\C, S_{\C})}$ and $T^{(\C', S_{\C'})}$ are equivalent as abelian categories.  
\end{defn}

This gives a new equivalence relation in the category $q\mathsf{ASmd}$. 
Surprisingly, there exist a Hamming scheme and a group-case association scheme which are Morita equivalent 
while their Bose-Mesner algebras are not Morita equivalent; see Proposition \ref{prop:group_Hamming} and 
Remark \ref{rem:example_equivalences}.

Theorem \ref{thm:Mitchell's_thm1} enables us to investigate tame schemoids with tools in the study of the module categories, for example derived functors. Thus in considering more general schemoids, one might expect a morphism between 
the given schemoid and a tame one. Furthermore, a restriction functor and its adjoint between functor categories of schemoids 
will be of great use in the study of schemoids.   

\begin{thm} \label{thm:adjoints} 
Let $(\D, S')$ be a tame schemoid and $u : (\C, S) \to (\D, S')$ a morphism of schemoids. 
Then the functor $u^* : \T^{(\D, S')} \to \T^{(\C, S)}$ induced by $u$ has a left adjoint 
$\text{\em Lan}_u : \T^{(\C, S)} \to \T^{(\D, S')}$ 
and a right adjoint $\text{\em Ran}_u : \T^{(\C, S)} \to \T^{(\D, S')}$.  
\end{thm}

Theorem \ref{thm:adjoints} allows us to define an appropriate cohomology group of a schemoid over a tame one. In fact, 
if we have a morphism $u : (\C, S) \to (\D, S')$ to a tame schemoid, then   
we can send a module in $\T^{(\C, S)}$ to the enough projective abelian category 
$\T^{(\D, S')}$ with the adjoint functor mentioned above.  Thus homological algebra on $\T^{(\D, S')}$ can be applicable to 
the study of the functor category $\T^{(\C, S)}$. 

\begin{defn} Let $(\D, S')$ be a tame schemoid with the set $\{id_x\}_{x \in ob \D}/\sim_{S'}$ finite. 
For a morphism $u : (\C, S) \to (\D, S')$ of schemoids and a functor $M \in \T^{(\C, S)}$, 
{\it (relative) schemoid cohomology} of $u$ with coefficients in $M$ is defined by 
$$
H^*( (\C, S) \stackrel{u}{\to} (\D, S') ; M) = H^*(u ; M) := 
\text{Ext}_{\T^{(\D, S)}}^*(\underline{\K}, \text{Ran}_uM). 
$$
\end{defn} 

We remark that if $(\C, S)$ is a schemoid which comes from a group $G$, then schemoid cohomology of the identity morphism on 
the schemoid is just the group cohomology of $G$; see Corollary \ref{cor:original_one} for more details.   

Let $u : (\C, S) \to (\D, S')$ be a morphism of schemoids whose target is tame.  
The adjoint pairs in Theorem \ref{thm:adjoints} induces adjoints between the category of chain complexes 
over the functor categories:
$$
\xymatrix@C35pt@R25pt{
Ch(\T^{(\D, S')}) \ar[r]^-{u^*}
& Ch(\T^{(\C, S)})   \ar@/^6mm/[l]^-{\text{Ran}_{u}}_-{\perp} 
 \ar@/_6mm/[l]_-{\text{Lan}_{u}}^-{\perp} 
}
$$
Indeed, the objectwise assignment of the functors $u^*$, $\text{Ran}_{u}$ and $\text{Lan}_{u}$ 
gives rise to the adjoint functors on the categories of chain complexes. 

We recall the cofibrantly generated model category structure of a module category 
described in \cite[Theorem 2.3.11]{Hovey} for example.  
The result \cite[Theorem 11.9.2]{Hir} due to Kan enables us to give a model category structure to  
$Ch(\T^{(\C, S)})$ by using that of 
$Ch(\K[\D]\text{-Mod})$ and adjoints mentioned above. 

\begin{thm} \label{thm:model_category_str} With the above notation,  
suppose further that $ob([\D])$ is finite. Then there is a cofibrantly generated model category structure on $Ch(\T^{(\C, S)})$ in which 
the weak equivalences are the maps that $\text{\em Ran}_{u}$ takes into weak equivalences in 
$Ch(\T^{(\D, S')})\cong Ch(\T^{[\D]})\simeq Ch(\K[\D]\text{\em -Mod})$. Moreover, $(u^*, \text{\em Ran}_{u})$ is 
a Quillen pair with respect to this model category structure.  
\end{thm}

Thus we obtain a Hochschild cohomology type invariant for Morita equivalence; see Theorem \ref{thm:Ext}. 

It seems that our proof of Theorem \ref{thm:adjoints} is not applicable to showing the existence of the left/right adjoint to 
the restriction functor of a morphism of schemoids {\it from} tame one; see Remark \ref{rem:inverse}. Then the choice of a morphism of schemoids from $(\C, S)$ {\it to} 
a tame schemoid is relevant to the consideration of a model category structure on $\T^{(\C, S)}$.  

Let $X$ and $Y$ be objects in an abelian category $\mathcal{A}$. Suppose that the category $Ch(\mathcal{A})$ of chain complexes 
over $\mathcal{A}$ has a model category structure. 
Then we recall the Ext group of $X$ by $Y$ which is defined by
$
\text{Ext}_{\mathcal{A}}^n(X, Y) := \text{Hom}_{\text{D}(\mathcal{A})}(X, Y[n]),
$ 
where $\text{D}(\mathcal{A})$ denotes the derived category of ${\mathcal A}$, namely the homotopy category 
$\text{Ho}(Ch(\mathcal{A}))$ of $Ch(\mathcal{A})$.   

\begin{rem}\label{rem:the_usual_Ext}
Let $\mathcal{A}$ be the category of (unbounded) chain complexes of left $R$-modules, where $R$ is a ring. 
When we consider the projective model structure on $Ch(\mathcal{A})$; see \cite[Section 2.3]{Hovey}, the Ext group 
$\text{Ext}_{\mathcal{A}}^*(X, Y)$ for $R$-modules $X$ and $Y$ is the usual one.   
\end{rem}

\begin{cor}\label{cor:cohomologies}
With the model category structure on $Ch(\T^{(\C, S)})$ defined in Theorem \ref{thm:model_category_str}, one has a natural isomorphism 
$$
H^*( (\C, S) \stackrel{u}{\to} (\D, S') ; M):=\text{\em Ext}_{\T^{(\D, S')}}^*(\underline{\K}, \text{\em Ran}_uM)
 \cong \text{\em Ext}_{\T^{(\C, S)}}^*(({\mathbb L}u^*) \underline{\K}, M)
$$
for every object $M$ in $\T^{(\C, S)}$, where ${\mathbb L}u^*$ denotes the total derived functor of the restriction $u^*$. 
\end{cor}


\begin{cor}\label{cor:original_one}
\text{\em (i)} For a group $G$ and a $\K G$-module $M$, the schemoid cohomology   
$H^*( S(G) \stackrel{id}{\to} S(G) ; M)$ is isomorphic to the group cohomology $H^*(G, M)$. \\
\text{\em (ii)} Let $\C$ be a small category with the set of morphism $mor(\C)$ finite. 
Then one has an isomorphism $H^*(\calK (\C) \stackrel{id}{\to} \calK(\C) ; M) \cong H^*(\C, M)$ for any $\K\C$-module 
$M$, where $H^*(\C, M)$ denotes the cohomology of $\C$ with coefficients in $M$; see \cite{B-W} for example. 
\end{cor}

As is seen below, even if a small category is equivalent to trivial one, 
the functor category in our context is not equivalent to $\T$ the trivial module category in general 
provided the small category admits a schemoid structure; see 
Example \ref{ex:N} and Remark \ref{rem:example_equivalences} again. 
Thus the functor categories of schemoids, which we deal with in this manuscript, 
are likely to provide new insights into categorical representation theory. 

\section{Tame schemoids}\label{Result}

We begin by recalling a schemoid arising from a groupoid. 
For a groupoid ${\mathcal H}$, we have 
a schemoid $\widetilde{S}({\mathcal H}) = (\widetilde{\mathcal H}, S)$, where $ob (\widetilde{\mathcal H}) = mor({\mathcal H})$ 
and 
$$
\text{Hom}_{\widetilde{{\mathcal H}}}(g, h) = \begin{cases}
\{(h, g)\}  & \text{if}  \ \  t(h) = t(g)  \\
\varnothing & \text{otherwise}. 
\end{cases}
$$
Here the partition $S=\{ \G_f \}_{f \in mor({\mathcal H})}$ is defined by $\G_f = \{(k, l) \mid k^{-1}l = f\}$. 
Observe that the composite of morphisms $(k, g)$ and $(g, f)$ in the category $\widetilde{\mathcal H}$ 
is defined by $(k, g)\circ (g, f) = (k, f)$. 

Let $\mathsf{Gpd}$ be the category of groupoids. 
Then we define functors $\widetilde{S}( \ ) : \mathsf{Gpd } \to q\mathsf{ASmd}$ and 
$\jmath : \mathsf{AS} \to q\mathsf{ASmd}$ 
by sending 
a groupoid ${\mathcal H}$ and an association scheme $(X, S)$ to 
$\widetilde{S}({\mathcal H})$ and $\jmath(X, S)$, respectively; see Example \ref{ex:2}.
One has a commutative diagram of categories 
$$
\xymatrix@C35pt@R20pt{
\mathsf{Gpd} \ar[r]^-{{\widetilde S}( \ )} & 
q\mathsf{ASmd} \ar@<1ex>[r]^-{U}_-{\top} 
& \mathsf{Cat}, \ar@<1ex>[l]^-{\calK}  \\
\mathsf{Gr} \ar[u]^\imath \ar[r]^-{S( \ )}  & \mathsf{AS} \ar[u]_{\jmath} 
}
\eqnlabel{add-1}
$$ 
where $\imath : \mathsf{Gr} \to \mathsf{Gpd}$ is the natural fully faithful embedding  
and the functor $S( \ )$ assigns group-case association schemes to groups; see \cite[Sections 2 and 3]{K-M} for more detail. 
Moreover, $\calK$ is a functor given by sending a small category to 
the discrete schemoid. Observe that the functor $\calK$ is the left adjoint to the forgetful functor $U$; see Example \ref{ex:1}. 

In order to define a tame schemoid, for a schemoid $(\C, S)$, we consider the following conditions T(i), T(ii) and T(iii). \\

\noindent
T(i): The schemoid $(\C, S)$ is unital, namely, for  $J_0 :=  \{id_x\}_{x \in ob\C}$,
$$
\{id_x\}_{x \in ob\C} = \displaystyle{\bigcup_{\alpha \in S, \alpha\cap J_0\neq \varnothing}}\alpha. 
$$

\noindent
T(ii):  For any $\sigma \in S$ and $f, g \in \sigma$, there exist $\tau_1$ and $\tau_2$ in $S$ such that 
$id_{s(f)}, id_{s(g)} \in \tau_1$ and $id_{t(f)}, id_{t(g)} \in \tau_2$. 

\begin{rem}\label{rem:T(ii)}
It follows from the proof of \cite[Lemma 4.2]{K-M} that the condition T(ii) necessarily holds for a unital schemoid. 
We recall the argument for the reader. Suppose that $f, g \in \sigma$, $s(f) \in \tau_1$ and $s(g) \in \tau_1'$. Then 
$p_{\tau_1\sigma}^\sigma \geq 1$ because $f$ is indeed in $\pi^{-1}_{\tau_1\sigma}(f)$.  Then 
there exists $u \in \tau_1$ and $h \in \sigma$ such that $u\circ h = g$. If $(\C, S)$ is unital, we see that 
$u= id_{s(g)}$ and hence $\tau_1' = \tau_1$. 
The same argument as in above yields that there is an element $\tau_2$ such that $id_{t(f)}, id_{t(g)} \in \tau_2$. 
\end{rem}

The third one is required to introduce a category $[\C]$ associated with a schmeoid $(\C, S)$,  
whose set of objects is defined by   
$$
ob [\C] = \{id_x\}_{x \in ob\C} \slash \sim_S \ =\{[x] \},  
$$ 
where we write $[x]$ for $[id_x]$.  Under the condition T(ii), for an element $\sigma \in S$, there exists a unique 
element $[x]$ in $ob[\C]$ such that $id_{s(f)} \in [x]$ for any $f \in \sigma$. In this case, we write $s(\sigma) \subset [x]$. 
Similarly, we write $t(\sigma) \subset [y]$ if $id_{t(f)} \in [y]$ for any $f \in \sigma$. 
Define a set of morphisms from $[x]$ to $[y]$ in the diagram $[\C]$ by 
$$
mor_{[\C]}([x], [y]) = \{\sigma \in S \mid s(\sigma) \subset [x], \ t(\sigma) \subset [y]\}.     
$$

\noindent
T(iii): For morphisms $[x] \stackrel{\sigma}{\longrightarrow} [y] \stackrel{\tau}{\longrightarrow} [z]$, 
there exist $f \in \sigma$ and 
$g \in \tau$ such that $s(g) = t(f)$. Moreover, there is a unique element $\mu = \mu(\tau, \sigma)$ in $S$ 
such that $p_{\tau \sigma}^{\mu} \geq 1$.  

\medskip
A schemoid $(\C, S)$ is called {\it tame} if the conditions T(i) and T(iii) hold.  

\begin{rem}
Let $(\C, S)$ and $(\C', S')$ be tame schemoids. It is readily seen that 
the product schemoid $(\C \times \C', S\times S)$ is tame. 
\end{rem}

We observe that a schemoid whose underlying category is the face poset of a simplicial complex is not tame in general; 
see Remark \ref{rem:not_tame} for such schemoids.   

\begin{lem} Let $(\C, S)$ be a tame schemoid. Then the diagram $[\C]$ is a category with the composite of morphisms 
defined by $\tau \circ \sigma = \mu(\tau, \sigma)$. 
\end{lem}

\begin{proof}
It suffices to show the associativity of the composite of morphisms. Consider composable morphisms 
$[x] \stackrel{\sigma_1}{\longrightarrow}[y]\stackrel{\sigma_2}{\longrightarrow}[z]\stackrel{\sigma_3}{\longrightarrow}[w]$. 
Suppose that $\mu = \sigma_2\circ \sigma_1$. Then the condition T(iii) implies that there exist composable morphisms 
$f \in \sigma_1$ and $g \in \sigma_2$ such that $gf$ is in $\mu$. By T(iii), we see that there exist $h \in \mu$ and $k \in \sigma_3$ such that $s(k) = t(h)$. Since $(\C, S)$ is a schemoid, it follows that 
\[
(\pi_{\sigma_2\sigma_1}^\mu)^{-1}(gf) \cong (\pi_{\sigma_2\sigma_1}^\mu)^{-1}(h). 
\]
Thus we have a diagram 
$\bullet \stackrel{f'}{\longrightarrow} \bullet \stackrel{g'}{\longrightarrow} \bullet  \stackrel{k}{\longrightarrow} \bullet$
for some $f' \in \sigma_1$ and $g' \in \sigma_2$ with $h=g'f'$. Let $\tau$ be an element in $S$ which contains $kg'f'$. 
The uniqueness in the condition T(iii) yields that 
$\sigma_3\circ (\sigma_2\circ \sigma_1) = \tau = (\sigma_3\circ \sigma_2) \circ \sigma_1$. This complets the proof.   
\end{proof}

\begin{lem} \label{lem:isomorphisms}
Let $(\C, S)$ be a tame schemoid. Then the category $\T^{(\C, S)}$ is isomorphic to the functor category 
$\T^{[\C]}$. 
\end{lem}

\begin{proof} 
Let $f$ be an element of $\sigma \in S$. Then we write $[f]$ for $\sigma$. 
We define functors 
\xymatrix@C35pt@R10pt{
\T^{(\C, S)}  \ar@<0.5ex>[r]^-{\Phi}_-{} &
\T^{[\C]}   \ar@<0.5ex>[l]^-{\Psi} } 
by 
$\Phi(G)([x]) = G(x)$, $\Phi(G)([f]) = G(f)$,  $\Psi(F)(x) = F([x])$ and $\Psi(F)(f)= F([f])$ for objects $G$ and $F$. 
Moreover, for morphisms $\eta : G \to G'$ in $\T^{(\C, S)}$ and $\nu : F \to F'$ in $\T^{[\C]}$, define 
$\Psi(\nu)x = \nu[x]$ and $\Phi(\eta)[x] = \eta[x]$, respectively. We see that $\Phi$ is a well-defined isomorphism with inverse $\Psi$. 
\end{proof}

It is known that the functor category $\T^\D$ is enough projective for any small category $\D$; see \cite[page 25]{Mitchell81} 
for example. Thus so is $\T^{(\C, S)}$ if $(\C, S)$ is tame. 

We have Mitchell's embedding theorem for a tame schemoid. 

\begin{thm} \label{thm:Mitchell's_thm} Let $(\C, S)$ be a tame schemoid. 
The category $\T^{(\C, S)}$ is embedded into the module category $\K[\C]\text{\em -Mod}$. 
Moreover,  $\T^{(\C, S)}$ is equivalent to the module category $\K(\C, S)\text{\em -Mod}$ 
if $ob [\C]$ is finite and 
the structure constant $p_{\tau \sigma}^\mu$ is less than or equal to $1$ for any $\tau, \sigma, \mu \in S$. 
\end{thm}

\begin{proof}
By Lemma \ref{lem:isomorphisms} and Mitchell's embedding theorem for a usual functor category, we have an 
embedding $\T^{(\C, S)} \cong \T^{[\C]} \to \K[\C]\text{-Mod}$. As for the latter half of the assertion, 
the embedding gives an equivalence of categories.  Moreover, the assumption yields that 
the algebra $\K[\C]$ is isomorphic to the Bose-Mesner algebra $\K(\C, S)$. We have the result. 
\end{proof}

\begin{prop}\label{prop:groupoids}
Let $\G$ be a groupoid. Then the associated schemoid 
$\widetilde{S}(\G)$ is tame and the structure constants are less than or equal to $1$. 
In particular $S(G)$ is tame for any group $G$. 
\end{prop}

\begin{rem}
The result \cite[Lemma 4.4]{K-M} implies that a {\it semi-thin} schemoid is tame; see \cite[Definition 4.1]{K-M} for the definition of a 
semi-thin schemoid. Moreover, for a semi-thin schemoid $(\C, S)$, the groupoid $\widetilde{R}(\C, S)$ constructed in \cite[Section 4]
{K-M} coincides with the category $[\C]$ mentioned above. 
By virtue of the result \cite[Theorem 4.11]{K-M}, we see that $[\widetilde{S}(\G)]\cong\G$ as a category for a groupoid $\G$.
\end{rem}

The schemoid $\widetilde{S}(\G)$ associated with a groupoid $\G$ is {\it thin} and hence semi-thin; 
see \cite[Definition 4.8, Theorem 4.11]{K-M} again. Thus we have Proposition \ref{prop:groupoids}. 
Here a more direct proof of the result is given. 

\begin{proof}[Proof of Proposition \ref{prop:groupoids}] 
The condition T(i) holds on $\widetilde{S}(\G)$. In fact, the schemoidd is unital; see \cite[Theorem 4.11]{K-M}. 

We consider maps $[f] \stackrel{\G_k}{\longrightarrow} [g] \stackrel{\G_l}{\longrightarrow} [h]$. 
For $(g', f') \in \G_k$ and $(h', g'') \in \G_k$, we choose a morphism $(g'{g''}^{-1}h', g')$. Then it follows that 
$(g'{g''}^{-1}h')^{-1}g' = l$ and hence the morphism is in $\G_l$. In the schemoid $\widetilde{S}(\G)$, 
structure constants are less than or equal to $1$. 
Indeed, if $p_{\G_l\G_m}^\sigma \neq 0$, then $\sigma = \G_{lm}$. This also implies that the condition T(iii) holds on 
$\widetilde{S}(\G)$.
\end{proof}

Let  $u, v:  (\C, S_{\C}) \to (\C', S_{\C'})$ be morphisms of schemoids and $\eta : u \Rightarrow v$ be a natural transformation between functors $u$ and $v$.  We say that $\eta$ {\it preserves the partition of identities} if 
$id_x \sim_S id_y$, then $\eta(x)$ and $\eta(y)$ are contained in the same element $\tau$ in $S_{\C'}$. 
If $(\C', S_{\C'})$ is the discrete schemoid $\T$ mentioned above, then this notion coincides with that of 
locally constant natural transformations. 

\begin{prop}\label{prop:equivalence} Let $u :  (\C, S_{\C}) \to (\C', S_{\C'})$ be a morphisms of schemoids. \\
\text{\em (i)} The restriction functor $u^* : \T^{\C'} \to  \T^{\C}$ induced by $u$  gives rise to a functor 
$u^* : \T^{(\C', S_{\C'})} \to  \T^{(\C, S_{\C})}$.  \\
\text{\em (ii)}   Suppose that $u$ 
is an {\it equivalence}; that is, there exist a morphism $w :  (\C', S_{\C'}) \to (\C, S_{\C})$
and natural isomorphisms $uw \Rightarrow 1$ and $wu \Rightarrow 1$ which preserve 
the partition of identities and so do the inverses. Then $(\C, S_{\C})$ and $(\C', S_{\C'})$ are Morita equivalent. 
\end{prop}

\begin{proof} 
(i) For any object $M$ in  $\T^{(\C', S_{\C'})}$ and for $f, g \in \sigma$, we see that 
$$
u^*M(f) = Mu(f) = Mu(g) = u^*M(g).
$$
Observe that $M$ is a morphism of schemoids. For a morphism $\alpha \in  \T^{(\C', S_{\C'})}(M, N)$, namely 
a locally constant natural transformation, it follows that 
$$
(u^*\alpha)(x) = \alpha u(x) = \alpha u(y) =(u^*\alpha)(y) 
$$
whenever $id_x\sim_S id_y$. In fact, $id_{u(x)} = u(id_x) \sim_{S'} u(id_y) = id_{u(y)}$.  \\
(ii) Let $\eta : uw \Rightarrow 1$ be the natural isomorphism. We will show that $\eta$ 
induces a natural isomorphism $\widetilde{\eta} : w^*u^*  \Rightarrow 1$ defined by 
$\widetilde{\eta}(M)(x) :=M\eta(x)$. 
Suppose that $id_x \sim_S id_y$. By assumption, there exists $\tau \in S$ such that 
$\eta(x)$ and $\eta(y)$ are in $\tau$. 
Since $M$ is a morphism of schemoids, it follows that 
$M\eta(x) = M\eta(y)$ and hence  $\widetilde{\eta}(M)$ is in $\T^{(\C', S_{\C'})}$. It is immediate that 
$\widetilde{\eta}$ gives an equivalence between $\T^{(\C, S_{\C})}$ and $T^{(\C', S_{\C'})}$.
\end{proof}

\begin{rem}\label{Mitchell's_correspondence}
Let $(C, S)$ be a schemoid with $mor(\C)$ finite. Suppose that  the schemoid 
$(\C, S)$ satisfies the condition T(i) and hence T(ii). For example, a schemoid arising from 
an association schemes is such one. We define a small $R$-linear category $R\text{-}[\C]$ by 
$ob(R\text{-}[\C]) = ob([\C])$ and 
$$
\text{Hom}_{R\text{-}[\C]}([x], [y]) := R\langle \text{Hom}_{[\C]}([x], [y]) \rangle, 
$$
namely the free $R$-module generated by the set $\text{Hom}_{[\C]}([x], [y])$. For morphisms 
$\sigma \in \text{Hom}_{[\C]}([y], [z])$ and $\tau \in \text{Hom}_{[\C]}([x], [y])$, the composite $\sigma \circ \tau$ of the morphisms
is defined  to be $\sum_{\mu}p^{\mu}_{\sigma \tau}\mu$. 
Let $\T^{R\text{-}[\C]}$ be the functor category of additive functors from $R\text{-}[\C]$ to $\T$. We define a pair $(\theta, \eta)$ of functors 
$$\xymatrix@C25pt@R20pt{
\theta  : \T^{R\text{-}[\C]}  \ar@<0.5ex>[r]&
R(\C, S)\text{-Mod}  : \eta  \ar@<0.5ex>[l]}, $$ 
which is so-called Mitchell's correspondence, by 
$\theta(F) = \bigoplus_{[x] \in ob([\C])}F([x])$ and $\eta (M)([x]) = [id_x]M$.  It is readily seen that 
$\theta$ is an embedding with left inverse $\eta$. Moreover, we see that $\theta$ is an equivalence with inverse $\eta$ if 
$ob([\C])$ is finite. Observe that $\T^{R\text{-}[\C]}$ is {\it not} functorial with respect to morphisms in $q\mathsf{ASmd}$ in general; 
see \cite[Section 6]{K-M} and  \cite[Section 6]{F_1}.
\end{rem}

\begin{prop}\label{prop:invariant}
Let $(\D, S')$ be a tame schemoid and $u' : (\C', S_{\C'}) \to (\D, S')$ a morphism of schemoids. 
Let $K :  (\C, S_{\C}) \to (\C', S_{\C'})$ be a morphism of schemoids whose underlying functor 
$K: \C \to \C'$ gives an equivalence of categories.  Suppose that the inverse to the functor $K$ is a morphism of schemoids. 
Then for any module $M$ in $\T^{(\C', S_{\C'})}$, one has an isomorphism
$$
H^*( (\C, S_{\C}) \stackrel{u'K}{\to} (\D, S') ; K^*M) \cong H^*( (\C', S_{\C'}) \stackrel{u'}{\to} (\D, S') ; M). 
$$
\end{prop}

\begin{proof}
For any module $N$ in $\T^{(\D, S')}$, 
we have isomorphisms 
\begin{eqnarray*}
\T^{(\D, S')}(N, \text{Ran}_{u'}M) &\cong & \T^{(\C', S_{\C'})}({u'}^*N, M) \\
&\maprightud{{K}^*}{\cong}&  \T^{(\C, S_{\C})}({K}^*{u'}^*N, {K}^*M) \\
&\cong & \T^{(\C, S_{\C})}({(u'K)}^*N, K^*M) \cong \T^{(\D, S')}(N, \text{Ran}_{u'K}K^*M). 
\end{eqnarray*}
The second isomorphism follows from Lemma \ref{lem:coincides} below. We have the result. 
\end{proof}

In Section \ref{Example}, we will obtain a Morita equivalence which is induced by a {\it non-equivalent} 
morphism between schemoids; see Remark \ref{rem:example_Hamming_one}.

\begin{lem}\label{lem:coincides}
Let $(\D, S')$ be a tame schemoid and $u : (\C, S) \to (\D, S')$ a morphism of schemoids. 
Suppose that one of modules $N$ and $M$ in $\T^{(\C, S)}$ is in the image of the 
functor $u^* : \T^{(\D, S')} \to \T^{(\C, S)}$. Then the Hom-set $\T^{(\C, S)}(M, N)$ coincides with 
$\T^{\C}(M, N)$.
\end{lem}

\begin{proof} 
We consider two adjoints 
\begin{eqnarray*}
\Gamma :  \T^{[\D]}(\Phi N, \text{Ran}_{(\pi\circ u)}M) &\cong&  \T^{\C}((\pi\circ u)^* \Phi N, M)=\T^{\C}(u^*N, M) \  \text{and}\\
\Omega : \T^{[\D]}(\text{Lan}_{(\pi\circ u)}M, \Phi N) &\cong&  \T^{\C}(M, (\pi\circ u)^* \Phi N)=\T^{\C}(M, u^*N).
\end{eqnarray*}
With explicit forms of left and right Kan extensions, it follows that the image of the adjoints consist of locally 
constant natural transformations. To see this, 
we recall the right Kan extension given by 
$$
\text{Ran}_{(\pi\circ u)}M ([d]) =\int_{c\in \C}M(c)^{[\D]([d], [u(c)])}
$$
for an object $[d] \in [\D]$; see \cite[X]{M}. 
The bijection $\Gamma$ is defined by the composite 
$$
\xymatrix@C35pt@R20pt{
\Gamma(f) : (\pi\circ u)^* \Phi N \ar[r]^-{(\pi\circ u)^*(f)} & (\pi\circ u)^*\text{Ran}_{\pi\circ u}M \ar[r]^-{counit} &
M
}
$$
for any $f : \Phi N \to \text{Ran}_{\pi\circ u}M$.  In view of the definition of $\pi$, 
we see that $(\pi\circ u)^*(f)$ is locally constant. 
Moreover, it follows that 
the counit $(\pi\circ u)^*\text{Ran}_{(\pi\circ u)}M \to M$ is locally constant. 
In fact, for any $c \in ob \C$, the map ${\text counit}_c$ is given by the projection 
$$
\text{Ran}_{(\pi\circ u)}M ([u(c)]) \subset \prod_{e\in \C} M(c)^{[u(c)] \stackrel{}{\to} [u(e)]} \to M(c)^{[u(c)] \stackrel{id}{\to} [u(c)]}
=M(c). 
$$
If $id_c \sim_S id_{c'}$, then $id_{u(c)}\sim_S id_{u(c')}$ and hence $[u(c)] = [u(c')]$. This implies that 
${\text counit}$ is locally constant. Observe that $M(c) =M(c')$ because $M$ is in $\T^{(\C, S)}$.
By definition, we see that $\T^{(\C, S)}(M, N)$ is a submodule of $\T^{\C}(M, N)$. 
Then $\T^{(\C, S)}(M, N)=\T^{\C}(M, N)$ if $M$ is in the image of the restriction functor. 
We leave the rest of this proof to the reader. 
\end{proof}

\begin{proof}[Proof of Theorem \ref{thm:adjoints}] We recall isomorphisms 
$\xymatrix@C35pt@R20pt{
\T^{(\D, S')}  \ar@<1ex>[r]^-{\Phi}_-{\cong} &
\T^{[\D]}   \ar@<1ex>[l]^-{\Psi} } $
and the projection functor $\pi : \D \to [\D]$ 
in the proof of Lemma \ref{lem:isomorphisms}. 
Let $F$ and $G$ be objects in $\T^{(\D, S')}$ and 
$\T^{(\C, S)}$, respectively. 
We observe that $\Phi(F)\circ \pi = F$. 
Lemma \ref{lem:coincides} and the existence of a left adjoint in functor categories yield a sequence of isomorphisms 
\begin{eqnarray*}
\T^{(\C, S)}(u^*F, G) &\cong& \T^{\C}(u^*F, G) \\
&\cong &  \T^{\C}((\pi\circ u)^*\Phi(F), G) \\
&\maprightud{\text{adjoint}}{\cong} &   \T^{[\D]}(\Phi(F), \text{Ran}_{\pi\circ u}G) \\
& \maprightud{\Psi}{\cong}& 
\T^{(\D, S')}(F, \Psi\text{Ran}_{\pi\circ u}G).  
\end{eqnarray*}
Thus it follows that $\text{Ran}_u:=\Psi\text{Ran}_{\pi\circ u} : \T^{(\C, S)} \to \T^{(\D, S')}$ is the right adjoint to 
$u^*$. By the same argument as above, we have the left adjoint to the restriction functor $u^*$. This completes the proof.
\end{proof}

\begin{rem}\label{rem:inverse}
Let $v : (\D, S') \to (\C, S)$ be a morphism of schemoids. The same argument as in the proof of 
Theorem  \ref{thm:adjoints} does not work well in showing the existence of the left/right adjoint of the restriction functor $v^* : \T^{(\C, S)} \to \T^{(\D, S')}$ in general even if $(\D, S')$ is tame. Indeed, suppose that 
$id_c \sim_S id_c'$. We claim that $\text{Lan}_vN (c) = \text{Lan}_vN (c')$ for any $N$ in $\T^{(\D, S')}$. 
Recall the left adjoint $\text{Lan}_v :  \T^\D \to T^\C$ is defined by 
$$
\text{Lan}_vN (c) = \displaystyle\int^{d \in \D} \C(v(d), c)\cdot N(d)
$$
for $N$ in $\T^{\D}$ and $c \in \C$; see \cite[X]{M}.  There is no relation between the hom-sets 
$\C(v(d), c)$ and $\C(v(d), c')$ in general. 
\end{rem}

\begin{proof}[Proof of Theorem \ref{thm:model_category_str}]
We first recall the cofibrantly generated model category structure of a module category 
described in \cite[Theorem 2.3.11]{Hovey} for example.  

Let $I$ and $J$ be the generating set of cofibrations and the generating set of trivial cofibrations of $Ch(\K[\D]\text{-Mod})$. 
That is, $I$ and  $J$ consist of maps $S^{n-1} \to D^n$, which are inclusions, and $0 \to D^n$ for $n \in {\mathbb Z}$, respectively.  
Here $D^n$ denotes the chain complex defined by $(D^n)_k = \K[\D]$ if $k = n$ or $k = n-1$ and $0$ otherwise with the only non trivial 
differential $d_n = id$. Moreover,  $S^{n-1}$ is the chain complex defined by $(S^{n-1})_{n-1} = \K[\D]$ and 
$(S^{n-1})_{k} = 0$ if $k \neq n-1$. 
Then we have to verify that (1) $u^*I$ and $u^*J$ permit the small object argument and that (2) $\text{Ran}_{u}$ takes relative 
$u^*J$-cell complexes to weak equivalences. The first one follows from the same argument as in 
\cite[Exmaple 2.1.6]{Hovey}. 

By making use of the description of the right adjoint $\text{Ran}_{u}$ in Theorem \ref{thm:Adj}, 
we see that the condition (2) holds. In fact, the domains of elements in $u^*J$ are trivial. Then every relative $u^*J$-cell complex 
has the form $j : A \to A \oplus \displaystyle{\bigoplus}_{\beta < \lambda} G_\beta$ for some ordinal $\lambda$, 
where $A$ is an appropriate chain complex and $G_\beta = u^*D^{n(\beta)}$. 
Since the nontrivial differential in each $G_\beta$ is the identity map, 
it follows that $\text{Ran}_u ({\oplus}_{\beta < \lambda} G_\beta)$ is contractible. This yields that 
$\text{Ran}_u(j) : \text{Ran}_uA \to \text{Ran}_u(A \oplus \displaystyle{\bigoplus}_{\beta < \lambda} G_\beta)$ is weak equivalence. 
Observe that $\text{Ran}_u$ is additive. We have the result. 
\end{proof}

\begin{proof}[Proof of Corollary \ref{cor:cohomologies}]
Since $(u^*, \text{Ran}_u)$ is a Quillen pair, it follows from \cite[Theorem 8.5.18]{Hir} that there exists a natural isomorphism 
$$ 
\text{Ext}_{\T^{(\C, S)}}^n({\mathbb L}u^* \underline{\K}, M) = 
\text{Hom}_{\text{D}(\T^{(\C, S)})}(u^*(\widetilde{C}\underline{\K}),  M[n]) \cong 
\text{Ext}_{\T^{(\D, S')}}^n(\underline{\K}, {\mathbb R}\text{Ran}_u M), 
$$ 
where $\widetilde{C}\underline{\K} \to \underline{\K}$ denotes a fibrant cofibrant approximation; 
see \cite[Definition 8.1.2, Proposition 8.1.3]{Hir}.  
Each object $M$ in $\T^{(\C, S)}$ is fibrant. In fact, $M \to 0$ has the right lifting property with respect to every element of 
$u^*J$.  
Then for the total derived functor ${\mathbb R}\text{Ran}_u$, 
we see that ${\mathbb R}\text{Ran}_uM = \text{Ran}_uM$. This completes the proof.
\end{proof}

\begin{proof}[Proof of Corollary \ref{cor:original_one}]
This follows from Theorem \ref{thm:Mitchell's_thm}. Observe that schemoids of 
the forms $S(G)$ and ${\mathcal K}(\C)$ are tame.  
\end{proof}

\section{Examples}\label{Example}

\begin{ex} Let $G$ be a group and $H$ a subgroup.  
Let $G/H$ denote the group-case association scheme whose underlying set is the homogeneous one $G/H$.   
By considering a normal subgroup $N$ containing $H$, 
we have a natural map $u : G/H \to G/N$. Then $G/N$ is a group and hence a tame schemoid. Therefore, 
for a functor $M \in \T^{(G/H)}$,   
schemoid cohomology $H^*(G/H \stackrel{u}{\to} G/N; M)$ is isomorphic to group cohomology of the form  
$H^*(G/N ; \text{Ran}_uM)$. 
\end{ex}

\begin{ex}\label{ex:N}
Let $L$ be a simplicial complex and ${\mathcal N}$ is 
a category whose objects are non-negative integers and which has the one arrow $i \to j$ if and only if $i\leq j$.  
The length of the arrow $i \to j$ is defined to be the difference $j -i$.  
In the category ${\mathcal N}$, lengths of arrows give a partition $len$ of the set of morphisms (arrows). 
Then we see that $({\mathcal N}, len)$ is a tame schemoid whose structure constants are less than or equal to $1$. 
Moreover, the Bose-Mesner algebra of this schemoid is isomorphic to the polynomial algebra 
$\K[\sigma_1]$. 

We have a morphism  $u: (P(L), S) \to ({\mathcal N}, len)$ of schemoids by "collapsing" the Hasse diagram 
of the face poset of $L$, where $(P(L), S)$ is the schemoid associated with $L$; see Remark \ref{rem:height_functor}. 
Thus the schemoid cohomolgoly of the morphism $u$ is considered by using the Koszul resolution of the constant functor 
$\underline{\K}$ as a $\K[\sigma_1]$-module. 
In fact, we have 
\begin{eqnarray*}
H^*(P(L) \stackrel{u}{\to} ({\mathcal N}, len); M) &\cong& \text{Ext}^*_{\K[\sigma_1]}(\underline{\K}, \text{Ran}_u M)\\
&\cong& H(\text{Hom}(\wedge(s^{-1}\sigma_1), \Psi\text{Ran}_{\pi\circ u} M); \delta)
\end{eqnarray*}
for any $M \in \T^{(P(L), S)}$; see Remark \ref{rem:the_usual_Ext}. 
It follows that the differential $\delta$ is defined by $\delta(f)(s^{-1}\sigma_1) = \sigma_1f(1)$. 

Let $n$ be the number of vertices of a simplicial complex $L$. Then we define a morphism 
$v: (P(L), S) \to ({\mathcal N}, len)^{\times n}$ of schemoids by $v(\phi) = (0, ...., 0)$ and 
$v(x_i ) =(0, ..., 0, 1, 0, ..., 0)$, where $1$ appears in the $i$th entry. Then the morphism $v$ also defines schemoid cohomology 
$H^*(v ; M)$. 
\end{ex}

\begin{prop}\label{prop:group_Hamming} Let $H(n, 2)$ be the Hamming scheme of binary codes with length $n$. More precisely, 
$H(n, 2)=( \{0, 1\}^{\times n}, \{T_0, T_1, ..., T_n\} )$, 
where $T_i$ denotes the set of the pair of words with the Hamming metric $i$. 
Then schemoids $\widetilde{S}({\mathbb Z}/2)$ and  $\jmath(H(n, 2))$ are Morita equivalent; see Section \ref{Result} for the functor 
$\widetilde{S}( \ )$. 
\end{prop}

We first consider the case of $H(2,2)$. 
The Hamming scheme gives a schemoid  $(\C, S)=\jmath(H(2,2))$ whose underlying category is pictured by 
the diagram  
$$
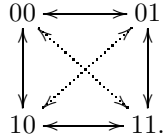
\xymatrix@C30pt@R30pt{
{00}  \ar@{<->}[r] \ar@{<.>}[rd]  \ar@{<->}[d] &  {01}  \ar@{<->}[d]\\
{10} \ar@{<->}[r] \ar@{<.>}[ru] &  {11}. 
}
$$
Here white arrows from a vertex to itself,  the black arrows and dots arrows are 
in $T_0$, $T_1$ and $T_2$, respectively. Observe that $S=\{T_0, T_1, T_2\}$. It is readily seen that 
$p_{T_0T_i}^{T_i} = 1 = p_{T_iT_0}^{T_i}$ for any $i$ and that $p_{T_1T_1}^{T_0}= 2$, 
$$
p_{T_1T_2}^{T_0} = 0 = p_{T_2T_1}^{T_0}, p_{T_1T_1}^{T_1} =0, 
p_{T_1T_2}^{T_1} = 1 = p_{T_2T_1}^{T_1}, p_{T_1T_1}^{T_2}= 1, p_{T_1T_2}^{T_2} = 0 = p_{T_2T_1}^{T_2}. 
$$ 
In particular, we have $T_1^2 = 2T_0 + T_2$ in the Bose-Mesner algebra of $H(2.2)$.
On the other hand, the schemoid $(\D, S')= \widetilde{S}({\mathbb Z}/2)$ is described by 
a diagram 
$$
\xymatrix@C30pt@R30pt{
{0}  \ar@<0.5ex>[r]^-{\alpha} &  {1} \ar@<0.5ex>[l]^-{\alpha^{-1}} ,   
}
$$
where the partition of the morphism is given by $\{\{id_0, id_1\}, \{\alpha, \alpha^{-1}\} \}$. 
For an object $M$ in $\T^{\jmath(H(2,2))}$, we see that  
$$M((01 \leftarrow 00))^2= M((00 \leftarrow 01)\circ (01 \leftarrow 00))= M(id_{00}) =id_{M(00)} \ \text{and} $$ 
$$M((01 \leftarrow 00))^2= M((11 \leftarrow 01)\circ (01 \leftarrow 00))= M(11 \leftarrow 00).$$ This yields that 
$M(T_2) =\{id_{M(00)}\}=M(T_0)$.

\begin{proof}[Proof of Proposition \ref{prop:group_Hamming}]
Let $u :  \widetilde{S}({\mathbb Z}/2)\to \jmath(H(n, 2))=:(\C, S)$ be a morphism of schemoids 
defined by $u(0) = 0\cdots 00$ and $u(1) = 0\cdots 01$. Then the morphism induces the functor $u^* : \T^{\jmath(H(2,n))} \to 
\T^{\widetilde{S}({\mathbb Z}/2)}$. 
Let $T_{even}$ and $T_{odd}$ denote $T_i$ for some $i$ even and $T_j$ for some $j$ odd, respectively. 
The same argument as in the case of $H(2,2)$ enables us to deduce that for any object $M$ in $\T^{\jmath(H(2,n))}$, 
$M(f) = M(0\cdots 00 \to 0\cdots 01) = M(0\cdots 01 \to 0\cdots 00)$ if $f$ is in $T_{odd}$ and $M(g) = M(id_{0\cdots 00})$ if $g$ is in $T_{even}$.

We identify the set of objects in $\C$ with ${\mathbb Z}/2^{\times n}$. Let $sgn : {\mathbb Z}/2^{\times n} \to {\mathbb Z}/2$ be a 
homomorphism defined by $sgn(i_1, ..., i_n)=i_1+\cdots +i_n$. 
With the homomorphism, 
we define a functor $v : \jmath(H(n, 2)) \to \widetilde{S}({\mathbb Z}/2)$ by 
$$
v(p) = \begin{cases}
0 & \text{if $sgn(p)=0$},  \\
1 &  \text{if $sgn(p)= 1$} 
\end{cases}
$$
for an object $p$ in $\jmath(H(n, 2))$. Thus it follows that 
$$
v(p \to q) = \begin{cases}
id_0 & \text{if $sgn(p)=0=sgn(q)$},  \\
\alpha &  \text{if $sgn(p)=0$ and $sgn(q)=1$}, \\
\alpha^{-1}  & \text{if $sgn(p)=1$ and $sgn(q)=0$}, \\
id_1 & \text{if $sgn(p)=1=sgn(q)$}.
\end{cases}
$$
We claim that $v$ preserves a partition.  Suppose that $p\to q$ is in $T_i$. Then we have $sgn(q) - sgn(p) =sgn(q - p) = i$ 
in ${\mathbb Z}/2$. This yields that $v(T_{odd}) \subset \{\alpha, \alpha^{-1}\}$ and 
$v(T_{even}) \subset \{id_0, id_1\}$. We see that $v$ is a morphism of schmeoids. 

It is immediate by definition that $v\circ u=1$ and hence $u^*\circ v^*$ is the identity functor on $\T^{\widetilde{S}({\mathbb Z}/2)}$;
see Proposition \ref{prop:equivalence}(i). 
Define a natural isomorphism $\theta : v^*\circ u^* \Rightarrow 1$ by 
$\theta(M)(a) = id_{M(a)}$. We have to verify that  $\theta(M)$ is a morphism in $\T^{\jmath(H(n, 2))}$ for any 
object $M \in \T^{\jmath(H(n, 2))}$. Observe that $\theta(M)$ is locally constant by definition. 
We consider a diagram 
$$
\xymatrix@C25pt@R20pt{
(v^*\circ u^*)M(a) \ar[r]^-{\theta(M)(a)} \ar[d]_{M(u(f'))=(v^*\circ u^*)M(f)}& M(a) \ar[d]^{M(f)}  \\ 
(v^*\circ u^*)M(b) \ar[r]_-{\theta(M)(b)} & M(b),                                                                     
}
$$
where 
$f' = \alpha$ if $f \in T_{odd}$ and $f' = id_0$ if $f \in T_{even}$. We see that $u(f'), f \in T_{odd}$ or 
$u(f'), f \in T_{even}$. It follows that the diagram is commutative and hence 
$\theta(M)$ is a morphism in $\T^{\jmath(H(2,2))}$.  
We claim that $\theta$ gives rise to a natural transformation from $v^*\circ u^*$ to the identity functor on $\T^{\jmath(H(n, 2))}$; that is, we show that  
a diagram 
$$
\xymatrix@C25pt@R20pt{
(v^*\circ u^*)(M) \ar[r]^-{\theta(M)} \ar[d]_{(v^*\circ u^*)(\alpha)}& M \ar[d]^{\alpha} \\
(v^*\circ u^*)(N) \ar[r]_-{\theta(N)} & N    
}
$$
is commutative for any $\alpha : M \to N$ in $\T^{\jmath(H(n,2))}$. To see this, consider a diagram 
$$
\xymatrix@C45pt@R20pt{
(v^*\circ u^*)(M)(a) \ar[r]^-{\theta(M)(a)=id} \ar[d]_-{(v^*\circ u^*)(\alpha)(a)} & M(a) \ar[d]^{\alpha(a)} \\
(v^*\circ u^*(N)(a) \ar[r]_-{\theta(N)(a) = id} & N(a)     
}
$$
for any $a \in ob \C$. 
Since $\alpha$ is locally constant and $id_{uv(a)}\sim_S id_{a}$, it follows that 
$\alpha(uv(a)) =\alpha(a)$ and then $(v^*\circ u^*)(\alpha)(a)=(uv)^*(\alpha)(a) =  \alpha(uv(a))$. 
Thus the diagram is commutative. 
It turns out  that $\widetilde{S}({\mathbb Z}/2)$ and  $\jmath(H(n, 2))$ are Morita equivalent. 
\end{proof}

\begin{rem}\label{rem:example_Hamming_one}
In the observation above, we see that the morphism $v$ of schemoids is an equivalence on the underlying categories with 
$u$ its inverse. In fact, for example, a natural isomorphism $\eta : uv \Rightarrow 1$ on $\jmath(H(n,2))$ 
is determined uniquely with the condition that 
$\eta_{11} : 00 =  uv(11) \to 11$ and  $\eta_{00} : 00 =  uv(00) \to 00$ because an association scheme is regarded as a complete graph. 
However, since $\eta_{00} \in T_0$, $\eta_{11}\in T_2$  and $id_{00}\sim_{S} id_{11}$, it follows that $\eta$ does not preserve 
the partition. Thus Proposition \ref{prop:equivalence} is not applicable to this case. 
  
Observe that an association scheme has an initial object if we consider it a category with the functor 
$\jmath : \mathsf{AS} \to q\mathsf{ASmd}$ and the forgetful functor $U : q\mathsf{ASmd} \to \mathsf{Cat}$ described in 
Section \ref{Result}. Therefore, each association scheme is equivalent to the trivial category as a category.  
\end{rem}

\begin{rem}\label{rem:example_equivalences}(i) The category $\T^{\jmath(H(2,2))}$ is {\it not} equivalent to the module category 
$\K(\jmath(H(2,2)))\text{-Mod}$. 
Suppose that the categories are equivalent. Then the Bose-Mesner algebras $\K(\jmath(H(2,2)))$ and  
$\K\widetilde{S}({\mathbb Z}/2) = \K({\mathbb Z}/2)$ are Morita equivalent. 
In fact, we see that  $\K(\jmath(H(2,2)))\text{-Mod}\simeq T^{\jmath(H(2,2))} \simeq 
T^{\widetilde{S}({\mathbb Z}/2)} \simeq \K({\mathbb Z}/2)\text{-Mod}$. The third equivalence follows from 
Theorem \ref{thm:Mitchell's_thm}. 
The algebra $\K(\jmath(H(2,2)))$ is commutative and a free $\K$-module of rank 
$3$.  On the other hand, the group ring $\K({\mathbb Z}/2)$ is of rank $2$ and commutative. 
Then the $0$th Hochschild homology groups of the two algebras are different from each other, which is a contradiction. \\  
(ii) The underlying categories of the schemoids $\widetilde{S}({\mathbb Z}/4)$ and $\mathcal{K}U(\jmath(H(2,2)))$ are 
the same as that of $\jmath(H(2,2))$. 
However, Theorem \ref{thm:Mitchell's_thm} enables us to deduce that the category $\T^{\widetilde{S}({\mathbb Z}/4)}$ is equivalent 
to the module category $\K({\mathbb Z}/4)\text{-Mod}$. 
The category $\T^{\mathcal{K}U(\jmath(H(2,2)))}$ is equivalent to $\T$. In fact, $\T^{\mathcal{K}U(\jmath(H(2,2)))}$ 
is isomorphic to the functor category 
$\T^{U(\jmath(H(2,2)))}$ because  ${\mathcal{K}U(\jmath(H(2,2)))}$ is tame. 
Moreover, the small category $U(\jmath(H(2,2)))$ is a groupoid with initial object.  
Then ${\widetilde{S}({\mathbb Z}/4)}$,  $\jmath(H(2,2))$ and $\mathcal{K}U(\jmath(H(2,2)))$ are {\it not} 
Morita equivalent one another.  
\end{rem}

A directed complete graph $K_n$ is regarded as a groupoid with initial object. As mentioned in Remark 
\ref{rem:example_equivalences} (ii), the functor category $\T^{\calK(K_n)}$ is equivalent to $M(n; R)\text{-Mod}$ and hence 
to $\T=R\text{-Mod}$. Thus $K_n$ is {\it Morita equivalent} to $R$. In this context, 
Proposition \ref{prop:group_Hamming} asserts that the Hamming schemes give a class of schemoids Morita equivalent to 
the group ring $R({\mathbb Z}/2)$, which is most small of all $R$-algebras bigger than $R$ itself. 

\section{An invariant for Morita equivalence}

In this section, we explore an invariant for Morita equivalence of schemoids such as the Hochschild cohomology of algebras.
We use terminology in \cite{P-N} while notations may be replaced with ours.  

Let $\text{Adj}({\mathcal B}, {\mathcal A})$ be the category of pairs of right and left adjoint functors between abelian categories 
${\mathcal A}_2$ and  ${\mathcal A}_1$. More precisely, the objects are adjoint pairs $(u_1, u_2) : 
\xymatrix@C35pt@R25pt{
{\mathcal B} \ar@<1ex>[r]^-{u_1}_-{\top} 
& {\mathcal A} \ar@<1ex>[l]^-{u_2}  }
$
and morphisms $\alpha=(\alpha_1, \alpha_2) : (u_1, u_2) \to (v_1, v_2)$ are pairs of natural transformations with 
$\alpha_1 : v_1 \Rightarrow u_1$ and $\alpha_2: u_2 \Rightarrow v_2$ which fit 
in the commutative diagram
$$
\xymatrix@C10pt@R20pt{
{\mathcal B}(v_2G, F) \ar[d]_{(\alpha_2G)^*}  \!\!& \cong  & \!\! {\mathcal A}(G, v_1F) \ar[d]^{(\alpha_1F)_*} \\
{\mathcal B}(u_2G, F)  \!\!& \cong & \!\! {\mathcal A}(G, u_1F), 
}
$$
where $G$ and $F$ are objects of ${\mathcal A}$ and ${\mathcal B}$, respectively. 
Let $U : \C_1^{op}\times \C_2 \to \T$ be a bifunctor. Then we define functors
$$\text{-}\otimes_{\C_1}U : \T^{\C_1} \to \T^{\C_2} \ \ \text{and} \ \
(U, \text{-})_{\C_2} : \T^{\C_2} \to \T^{\C_1}$$ by 
$(F\otimes_{\C_1}U)(a)= F\otimes_{\C_1}U(\text{-}, a)$ and $(U, G)_{\C_2}(b)= \text{Hom}_{\T^{(\C_2, S_2)}}(U(b, \text{-}), G)$, respectively. Indeed, we can verify the well-definedness for the functors directly.

\begin{lem} Let $(\C_1, S_1)$ and $(\C_2, S_2)$ be schemoids and $U$ an object 
in the functor category $\T^{(\C_1^{op}\times \C_2, S_1 \times S_2)}$.  
Then the restrictions of the functors  $\text{-}\otimes_{\C_1}U$ and $(U, \text{-})_{\C_2}$ mentioned above to functor 
categories of schemoids give rise to functors 
$$\text{-}\otimes_{\C_1}U : \T^{(\C_1, S_1)} \to \T^{(\C_2, S_2)} \ \ \text{and} \ \
(U, \text{-})_{\C_2} : \T^{(\C_2, S_2)} \to \T^{(\C_1, S_1)}.$$ 
\end{lem} 

\begin{proof}
We verify that $\text{-}\otimes_{\C_1}U :  \T^{(\C_1, S_1)} \to \T^{(\C_2, S_2)}$ is well defined. 
Let $f : a\to b$ and $g : a' \to b'$ be morphisms in $\C_2$ with $f\sim_{S_2} g$. Since $U$ is in $\T^{(\C_1^{op}\times \C_2, S_1 \times S_2)}$, it follows that $f_* : U(d, a) \to U(d, b)$ coincides with $g_* : U(d, a') \to U(d, b')$ for any $d \in ob(\C_1)$. 
In particular, $U(d, a) = U(d, a')$ and $U(d, b) = U(d, b')$. Observe that 
by definition, for $F$ in $\T^{(\C_1, S_1)}$ and an object $a$ in $\C_2$, 
$$
\xymatrix@C10pt@R20pt{
(F\otimes_{\C_1}U )(a) = F\otimes_{\C_1}U( \text{-}, a) = \text{coeq}({\displaystyle
\coprod_{\varphi : d \to d'}}Fd'\otimes U(d,a) \ar@<2.0ex>[rr]^-{\varphi^*}\ar@<0.5ex>[rr]_-{\varphi_*}& &
{\displaystyle\coprod_{d}}Fd \otimes U(d, a)
).
}
$$
Then we see that $f_* = g_*  : (F\otimes_{\C_1}U )(a) = (F\otimes_{\C_1}U )(a') \to (F\otimes_{\C_1}U )(b') = (F\otimes_{\C_1}U )(b)$.  
For any objects $a$ and $b$ with $id_a \sim_{\C_2} id_b$, it follows that $U(id_d, id_a) = U(id_d, id_b)$ 
and hence $U(d, a) = U(d, b)$. Thus one obtains a commutative diagram 
$$
\xymatrix@C35pt@R20pt{
F\otimes_{\C_1} U(a) \ar[r]^-{(\eta\otimes_{\C_1} U)_a}\ar@{=}[d]& G\otimes_{\C_1}U(a) \ar@{=}[d]\\
F\otimes_{\C_1} U(b) \ar[r]^-{(\eta\otimes_{\C_1} U)_b}& G\otimes_{\C_1}U(b) 
}
$$
for each morphism $\eta : F \Rightarrow G$ in $\T^{(\C_1, S_1)}$. Hence, $\eta\otimes_{\C_1} U$ is locally constant and then 
it is a morphism in $\T^{(\C_2, S_2)}$. 

For an object $G$ in $\T^{(\C_2, S_2)}$ and a map $f : a \to b$ in $\C_2$, the assignment 
\begin{eqnarray*}
f_*= (f^*)^* : (U, G)_{\C_2}(a) \!\!\!\!\!&=&\!\!\!\!\!\text{Hom}_{\T^{(\C_2, S_2)}}(U(a, \text{-}), G)\\
  &\longrightarrow& \text{Hom}_{\T^{(\C_2, S_2)}}(U(b, \text{-}), G)=(U, G)_{\C_2}(b)
\end{eqnarray*}
is well defined. To see this, suppose that $id_x\sim_{S_2} id_y$. 
Then for any morphism $\eta$ in $\text{Hom}_{\T^{(\C_2, S_2)}}(U(a, \text{-}), G)$, it is readily seen that 
$((f^*)^*\eta)_x = \eta_x f^* = \eta_y f^* =((f^*)^*\eta)_y$ because $\eta$ is locally constant. 
Let $f : a\to b$ and $g : a' \to b'$ be morphisms in $\C_1$ with $f\sim_{S_1} g$. Since 
$f^* = g^* : U(a, \text{-}) = U(a', \text{-})\to U(b, \text{-}) = U(b', \text{-})$, it follows that 
$((f^*)^*\eta)_x = ((g^*)^*\eta)_x$
for 
$\eta$ in $\text{Hom}_{\T^{(\C_2, S_2)}}(U(a, \text{-}), G)$ and $x \in ob(\C_2)$. In order to show 
$(U, \text{-})_{\C_2} : \T^{(\C_2, S_2)} \to \T^{(\C_1, S_1)}$ is a functor, we have to prove that $(U, \eta)_{\C_2}$ is locally constant 
for a morphism $\eta : G \Rightarrow F$. Assume that $id_a \sim_{S_1}id_b$. Then $U(a, \text{-}) = U(b, \text{-})$. 
Therefore,  one has a commutative diagram  
$$
\xymatrix@C40pt@R20pt{
\text{Hom}_{\T^{(\C_2, S_2)}}(U(a, \text{-}), G) \ar[r]^-{(U, \eta)_{\C_2}(a)}\ar@{=}[d]& 
\text{Hom}_{\T^{(\C_2, S_2)}}(U(a, \text{-}), F) \ar@{=}[d]\\
\text{Hom}_{\T^{(\C_2, S_2)}}(U(b, \text{-}), G) \ar[r]_-{(U, \eta)_{\C_2}(b)} &
\text{Hom}_{\T^{(\C_2, S_2)}}(U(b, \text{-}), F) 
}
$$
This completes the proof. 
\end{proof}

\begin{thm}\label{thm:Adj} Let $(\C_1, S_1)$ be a tame schemoid. Then  
$$
\Phi :\T^{(\C_1^{op}\times \C_2, S_1 \times S_2)} \to \text{\em Adj}(\T^{(\C_2, S_2)}, \T^{(\C_1, S_1)})
$$
defined by $\Phi(U) = ((U, \text{-})_{\C_2}, \text{-} \otimes_{\C_1}U)$ is well defined and an equivalence of categories.  
\end{thm}

\begin{proof} For an object $A$ in $(\C_1, S_1)$, we can define a projective object $h^A$ in $\T^{(\C_1, S_1)}$ 
by $h^A(B) = \K\text{Hom}_{[C_1]}([A],[B])$; see Lemma \ref{lem:isomorphisms} and \cite{Mitchell81}. 
Thus, the proof is verbatim the same as that in \cite[Corollary 2.2]{P-N}.
\end{proof}

\begin{thm}\label{thm:Ext}
Let $u : (\C_1, S_1) \to (\C_2, S_2)$, $w : (\C_1, S_1) \to (\D, S')$ and $w' : (\C_2, S_2) \to (\D, S')$ be 
morphisms of schemoids with $w'u = w$. Assume that $(\D, S')$ is tame. 
If $u$ induces an equivalence between $\T^{(\C_2, S_2)}$ and 
$\T^{(\C_1, S_1)}$; that is, $(C_1, S_2)$ and $(C_2, S_2)$ are Morita equivalence, then so are 
$T^{(\D^{op}\times \C_1, S'\times S_1)}$ and $T^{(\D^{op}\times \C_2, S'\times S_2)}$. In consequence, 
for any object $h$ in $\T^{(\D^{op}\times \D, S'\times S')}$, 
one has an isomorphism
$$
\text{\em Ext}^*_{T^{(\D^{op}\times \C_1, S'\times S_1)}}((1\times w)^*h, (1\times w)^*h) \cong 
\text{\em Ext}^*_{T^{(\D^{op}\times \C_2, S'\times S_2)}}((1\times w')^*h, (1\times w')^*h).
$$
\end{thm}

\begin{proof}
We have a diagram 
$$
\xymatrix@C15pt@R5pt{ & \T^{(\D^{op}\times \C_1, S'\times S_1)} \ar[r]^-\Phi_-{\simeq} & \text{Adj}(\T^{(\C_1, S_1)}, \T^{(\D, S')})\\ 
\T^{(\D^{op}\times \D, S'\times S')} \ar[ru]^{(1\times w)^*} \ar[rd]_{(1\times w')^*}& \\
  & \T^{(\D^{op}\times \C_2, S'\times S_2)} \ar[uu]_{(1\times u)^*} \ar[r]_-{\Phi}^-{\simeq}& \text{Adj}(\T^{(\C_2, S_2)}, \T^{(\D, S')}) 
  \ar[uu]_{(\eta, u^*)}
}
$$
in which the triangle is commutative and the square is commutative up to natural isomorphism, 
where $\eta$ denotes the inverse to $u^*$. The functor $(\eta, u^*)$ is an equivalence and hence so is $(1\times u)^*$. 
We have the result. 
\end{proof}

\begin{rem}
In Theorem \ref{thm:Ext}, suppose further that $ob[\D]$ is finite. Then Theorem \ref{thm:Mitchell's_thm} allows us to deduce that  $\T^{(\D^{op}\times \D, S'\times S')}$ is equivalent to the module category $\K[\D]^{op}\otimes \K[\D]\text{-Mod}$.  
Therefore, if $(\C_1, S_1)=(\D, S')$, $w$ is identity and $h$ is the functor corresponding to the bimodule $\K[\D]$, then
the Ext group in Theorem  \ref{thm:Ext} is isomorphic to the Hochschild cohomology 
$HH^*(\K[\D])$. 

Let $w : (\C_1, S_1) \to (\D, S')$ be a morphism of schemoids, which induces a Morita equivalence between the schemoids. 
We define a functor $\widetilde{h} : (\D^{op}\times \C_1, S'\times S_1) \to \T$ by 
$\widetilde{h}(a, b) = (id_{[a ,w(b)]})_*\K[\D]$ for any object  $(a, b)$ and  
$\widetilde{h}(f, g) : \widetilde{h}(a, b) = (id_{[a ,w(b)]})_*\K[\D] \to \widetilde{h}(c, d) = (id_{[c ,w(d)]})_*\K[\D]$ by 
$\widetilde{h}(f, g)=(f, w(g))_*$ for $(f, g): (a, b) \to (c, d)$ in $\D^{op}\times \C_1$.  
Then it follows that $(1\times w)^*h = \widetilde{h}$ and hence  
$\text{Ext}^*_{T^{(\D^{op}\times \C_1, S'\times S_1)}}((1\times w)^*h, (1\times w)^*h) \cong  
HH^*(\K[\D])$. 
\end{rem}

\medskip
\noindent
{\it Acknowledgements.}
The authors are grateful to Yasuhide Numata for showing a fundamental idea for constructing  schemoids from posets, 
without which they could not be aware the existence of a new schemoid which an abstract simplicial complex defines. 
The first author would like to thank Dai Tamaki for pointing out that the linear category in 
Remark \ref{Mitchell's_correspondence} is also considered naturally. 
This research was partially supported by a Grant-in-Aid for Scientific Research HOUGA 25610002 
from Japan Society for the Promotion of Science.
 
\section{Appendix 1}
In this section, we construct a quasi-schemoid from a poset. It is mentioned that the construction 
of schemoids described here is generalized to that for a small acyclic category in \cite{Numata}.  

Let $X$ be a set and $\Theta$ a subset of the power set $2^X$. We regard $\Theta$ as a category whose objects are elements of itself and whose morphisms are inclusions. 

\begin{lem}\label{lem:construction}\text{\em (cf. \cite[Lemma 1.1]{Numata})}
Let $D$ be the set consisting of the differences $V\backslash U$ for all morphisms (inclusions) $U \to V$ in $\Theta$. Define 
a partition $S$ of the set of morphisms in $\Theta$ by $S = \{\widetilde{\sigma}\}_{\sigma \in D}$, where 
$\widetilde{\sigma} = \{ i : U \to V \mid V\backslash U = \sigma\}.$ Then $(\Theta, S)$ is schemoid. 
\end{lem}

\begin{proof}
Let $i : U \to V$ be in $\widetilde{\sigma}$. Suppose that 
$i = k\circ l$ for $l : U \to W$ and $k : W \to V$. If $l\in \widetilde{\tau}$ and $k \in \widetilde{\mu}$, then $W = U\sqcup \tau$ and 
$V = W \sqcup \mu$ and hence $\sigma =  \tau \sqcup \mu$. Assume further that $i = k'\circ l'$, where 
$l' : U \to W' \in \widetilde{\tau}$ and $k' : W' \to V \in \widetilde{\mu}$. Then $W = U \sqcup \tau = W'$. Thus $l=l'$ and $k=k'$. This implies that 
$$
p^{\widetilde{\sigma}}_{\widetilde{\mu} \widetilde{\tau}}= \begin{cases}
1 & \text{if $\sigma =  \tau \sqcup \mu$}  \\
0 & \text{otherwise}. 
\end{cases}
$$
We have the result. 
\end{proof}

\begin{rem}\label{rem:height_functor}
Let $(\Theta, S)$ be the schmemoid mentioned above obtained by a set $X$. 
Suppose that $X$ is finite. Then we define a functor 
$u : \Theta \to {\mathcal N}$ by $u(U)= \sharp U$. 
It is readily seen that $u$ gives rise to a morphism 
$u : (\Theta, S) \to ({\mathcal N}, len)$ of schemoids. Thus Theorem \ref{thm:model_category_str}
allows us to give a model category structure to the category of chain complexes $Ch(\T^{(\Theta, S)})$. 
Thus we have a derived category of the form $\text{D}(\T^{(\Theta, S)})$. 
\end{rem}

Let $K$ be an abstract simplical complex and $P(K)$ its face poset. We consider the face poset 
a small category with an initial object $\varnothing$ as usual. Let $S_K$ be a partition of the set of morphisms defined by 
$$S_K = \{\widetilde{\sigma}\}_{\sigma \in K \cup \{\varnothing\}},$$ where 
$ \widetilde{\sigma} = \{ \alpha : \mu \to \tau \mid \tau\backslash \mu = \sigma\}$. Lemma \ref{lem:construction} yields that 
$(P(K), S_K)$ is a schemoid. It is called the {\it quasi-schemoid associated with} $K$. 

\begin{rem}\label{rem:not_tame}
Consider the simplicial complex $K$ with two vertices $1$, $2$ and no $1$-simplex. 
With the notation in Section 3,  we have a sequence  $[\varnothing] \stackrel{\widetilde{\{1\}}}{\longrightarrow} [\{1\}] = [\varnothing]
\stackrel{\widetilde{\{2\}}}{\longrightarrow} [\{2\}]$ in the diagram $[P(K)]$.  Since $t(\alpha) = \varnothing 
\neq s(\beta)$ for any $\alpha \in \widetilde{\{1\}}$ and $\beta \in \widetilde{\{2\}}$, it follows that 
$P(K)$ is not tame.
\end{rem}

We call a simplicial map $f : K \to L$ {\it non-degenerate} if $f(i) \neq f(j)$
whenever $\{i, j\}$ in $K$; see \cite[Definition 2.7]{B-P}.

\begin{lem} \label{lem:morphisms}
Let $f : K\to L$ be a simplicial map. 
Then the poset map $P(f) : P(K) \to P(L)$  
induces a well-defined morphism 
$P(f) : (P(K), S_K) \to (P(L), S_L)$ of schemoids if and only 
if $f$ is componentwise non-degenerate or constant.   
\end{lem}

\begin{proof}
Assume that $f$ is componentwise non-degenerate or constant. Let $\widetilde{\sigma}$ be an element in $S_K$ 
with $\sigma \neq \varnothing$. We show that  
there exists an element $\widetilde{\sigma'}$ in $S_L$ such that $P(f)(\widetilde{\sigma})$ is included in $\widetilde{\sigma'}$. 
Observe that there is a unique connected component $K_0$ such that $\sigma \in K_0$. Then the assertion we claim is immediate for the case $f$ is constant on $K_0$. 
Consider the case where $f$ is non-degenerate on $K_0$. 
For an element $\alpha : \mu \to \tau$ in $\widetilde{\sigma}$, we see that $\tau = \mu \sqcup \sigma$.
If $f(\mu) \cap f(\sigma) \neq \varnothing$, then there exist elements $i \in \mu$ and $j \in \sigma$ such that $f(i)=x=f(j)$ 
for some $x$ in $f(\mu) \cap f(\sigma)$, which is a contradiction. 
This implies that $f(\tau)=f(\mu \sqcup \sigma)=f(\mu) \sqcup f(\sigma)$.
Therefore,  we have $P(f)(\sigma) \subset \widetilde{f(\sigma)}$.

Assume that $f$ is not non-degenerate. If $f$ is not constant on a connected component $K_0$, then $|K_0| \ge 2$.
Since $f$ is not non-degenerate on $K_0$, there exists an element $\{ i, j \} \in K$ such that $f(i) = f(j)$.
The morphisms $\alpha_1 : \varnothing \to \{ i \}$ and $\alpha_2 : \{ j \} \to \{ i, j \}$ in $P(K)$ are in the partition $\widetilde{\{ i \}}$.
It follows that $P(f)(\alpha_1)$ belongs to $\widetilde{\{ f(i) \}}$ while $P(f)(\alpha_2)$ is in $\widetilde{\varnothing}$.
Hence the poset map $P(f)$ is not a morphism of schemoids. We have the result.
\end{proof}

The following result describes a strong connection between the Stanley-Reisner ring of a simplicial complex and the Bose-Mesner algebra of the schemoid associated with the simplicial complex.  

\begin{prop}\label{prop:B-MvsS-R} Let $K$ be a finite simplicial complex. Then there exists an isomorphism 
$\alpha_K : \K[K]/(x_i^2) \stackrel{\cong}{\to} \K(P(K), S_K)$ such that $\alpha_K(x_i) =\widetilde{\{i\}}$. 
\end{prop}

\begin{proof}
Since $\K(P(K), S_K)$ is generated by the elements $\widetilde{\{i\}} \in S_K$. We define epimorphism 
of algebras $\alpha' : \K[x_j] \to \K(P(K), S_K)$ by $\alpha'(x_j) = \widetilde{\{j\}}$. Suppose that $\{{i_1}, ...., {i_l}\}$ is not in 
$K$. We then see that $\widetilde{\{{i_1}\}}\cdots \widetilde{\{{i_l}\}} = 0$ in the category algebra $\K P(K)$ and hence in 
the Bose-Mesner algebra $\K(P(K), S_K)$. Moreover, we have $\widetilde{\{i\}}^2= 0$ in $\K(P(K), S_K)$. This yields that 
$\alpha'$ induces a well-defined epimorphism $\alpha_K$ of algebras. It is readily seen that $\text{rank} \K[K]/(x_i^2)$ and 
$\text{rank} \K(P(K), S_K)$ coincide with the number of the simplexes in $K$. This completes the proof. 
\end{proof}

Proposition \ref{prop:B-MvsS-R} allows us to deduce the following result. 

\begin{assertion} \label{assertion:simplicial_complexes} Let $K$ and $L$ be finite simplicial complexes. 
Then the following conditions are equivalent.
\begin{itemize}
\item[(i)] $K$ is isomorphic to $L$ as a simplicial complex.
\item[(ii)] $(P(K), S_K)$ is isomorphic to $(P(L), S_L)$ as a quasi-schemoid. 
\item[(iii)] $(P(K), S_K)$ is homotopy equivalent to $(P(L), S_L)$. 
\item[(iv)] The Bose-Mesner algebra of $(P(K), S_K)$ is isomorphic to that of $(P(L), S_L)$ as an algebra. 
\item[(v)] The Stanley-Reisner ring of $K$ is isomorphic to that  of $L$ as an algebra.  
\end{itemize}
\end{assertion}

\begin{proof}
The Bose-Mesner algebra $\K(P(K), S_K)$ is isomorphic to $\K[K]/(x_i^2)$ 
as an algebra. Thus it follows from the results Proposition \ref{prop:B-MvsS-R} and  \cite[Example 5.28, page 178]{B-G} that 
the conditions (i), (ii), (iv) and (v) are equivalent. 

We consider the implication from (iii) to (ii). In order to prove the claim, it suffices to show that if
$F \sim 1$, then $F=1$ for a morphism $F : (P(K), S_K) \to (P(K), S_K)$;  
see Definition \ref{defn:Homotopy}. 
Let $H :  (P(K), S_K) \times I \to (P(K), S_K)$ be a homotopy from $F$ to $1$. 
Then we have a commutative diagram
$$
\xymatrix@C35pt@R25pt{
{H(s(\varphi), 0)}  \ar@{->}[r]^{H(1_{s(\varphi)}, u)} \ar[dr]^{H(\varphi,u)}   \ar@{->}[d]_{F(\varphi)=H(\varphi, 1_0)}  &  {H(s(\varphi), 1)}  \ar@{->}[d]^{H(\varphi, 1_1)=1(\varphi)=\varphi}\\
{H(t(\varphi), 0)} \ar@{->}[r]_{H(1_{t(\varphi)}, u)}  &  {H(s(\varphi), 1)}
}
$$
for a morphism $\varphi$ in $P(K)$, where $u :  0 \to 1$ is the only non-trivial morphism in the small category $I$ with objects $0$ and 
$1$.  Since $1_{s(\varphi)}$ and $1_{t(\varphi)}$ are in $\widetilde{\varnothing}$, it follows that 
$H(1_{s(\varphi)}, u)\sim_{S_K} H(1_{t(\varphi)}, u)$. Let $\varphi : \varnothing \to \{i\}$ be a morphism in $P(K)$. The diagram gives a commutative one 
$$
\xymatrix@C35pt@R25pt{
F(\varnothing)  \ar@{->}[r]^{H(1_{s(\varphi)}, u)}  \ar@{->}[d]_{F(\varphi)} &  \varnothing \ar@{->}[d]^{\varphi}\\
F(\{i\}) \ar@{->}[r]_{H(1_{t(\varphi)}, u)} &  \{i\}.}
$$
Since $\varnothing$ is an initial object, we have $F(\varnothing)=\varnothing$. Then 
$H(1_{s(\varphi)}, u) = id$ and hence $H(1_{t(\varphi)}, u) = id$.  This yields that $F(\varphi) = \varphi$. 
By induction on the cardinality of a simplex, we see that $F(\psi) = \psi$ for a morphism  
$\psi : \{i_1, ..., i_s\} \to \{i_1, ..., i_s, i_{s+1}\}$ in $P(K)$. 
Since $F$ is a functor, it follows that $F(\varphi)= \varphi$ for any morphism $\varphi$ in $P(K)$.

Assume that there exists a homotopy $H :  (P(K), S_K) \times I \to (P(K), S_K)$ from $1$ to $F$. 
Let $\varphi : \varnothing \to \{j\}$ be a map in $K(P)$. We have a commutative diagram 
$$
\xymatrix@C35pt@R25pt{
\varnothing  \ar@{->}[r]^-{H(1_{s(\varphi)}, u)}  \ar@{->}[d]_{\varphi} &  F(\varnothing) \ar@{->}[d]^{F(\varphi)}\\
\{j\} \ar@{->}[r]_-{H(1_{t(\varphi)}, u)} &  F(\{j\}). }
$$
Suppose that $F(\varnothing) \neq \varnothing$. Then $H(1_{s(\varphi)}, u)$ is not identity. Let $\tau$ be a maximal simplex in 
$P(K)$ and $\psi : \varnothing \to \tau$ a map in $P(K)$. We obtain a commutative diagram 
$$
\xymatrix@C35pt@R25pt{
\varnothing  \ar@{->}[r]^-{H(1_{s(\psi)}, u)}  \ar@{->}[d]_{\psi} &  F(\varnothing) \ar@{->}[d]^{F(\psi)}\\
\tau\ar@{->}[r]_-{H(1_{t(\psi)}, u)} &  F(\tau). }
$$
Since $1_{s(\varphi)} \sim_{S_K} 1_{s(\psi)} \sim_{S_K} 1_{t(\psi)}$, we see that  $H(1_{t(\psi)}, u)$ is not identity. 
Therefore, $\tau \lneq F(\tau)$, which is a contradiction. Thus it follows that $F(\varnothing) =\varnothing$. 
By the same argument as in the case $F\sim 1$, namely by induction on the cardinality of a simplex, we have $1 = F$.  
\end{proof}

For a simplicial map $\varphi: K \to L$, 
we define a map $\varphi^* : \K[L] \to \K[K]$ between the Stanley-Reisner rings by 
$$
\varphi^*(y_j) = \sum_{i \in \varphi^{-1}(j)} x_i
$$ 
where $x_i$ and $y_j$ are generaters of  $\K[K]$ and $\K[L]$, respectively.  We see that $\varphi^*$ is a well-defined algebra map; see \cite[Proposition 3.4]{B-P} for example. Moreover, it follows that $\varphi^*$ induces a well-defined algebra map 
$$
\overline{\varphi^*} : \K[L]/(y_j^2) \to \K[K]/(x_i^2)
$$
if $\varphi$ is non-degenerate. In fact, we see that $\varphi^*(y_j^2)= \sum_{k,l \in \varphi^{-1}(j)}x_kx_l$. Suppose that $\{k, l\} \in L$ for some $k, l \in \varphi^{-1}(j)$. Then $\varphi(k) = \varphi(l)$, which is a contradiction because $\varphi$ is a non-degenerate simplicial map. Thus $\{k, l\}$ is not in $K$ for any $i, j \in \varphi^{-1}(j)$ and then $\varphi^*(y_j^2)=0$ in $\K[K]/(x_i^2)$. 

Let $\varphi : K \to L$ be a non-degenerate simplicial map.  Lemma \ref{lem:morphisms} implies that 
$P(\varphi) : (P(K), S_K) \to (P(L), S_L)$ is a morphism of quasi-schemoids. 
We define a map $P(\varphi)^* : \K(P(L), S_L) \to \K(P(K), S_K)$ of $\K$-modules by 
$$
P(\varphi)^*(\widetilde{\tau}) = \sum_{P(\varphi)(\widetilde{\sigma}) = \widetilde{\tau}, 
\widetilde{\sigma} \in S_K} \widetilde{\sigma}.
$$

\begin{prop} Let $\alpha_K$ and $\alpha_L$ be the isomorphisms described in Proposition \ref{prop:B-MvsS-R}. 
Then the map $P(\varphi)^*$ is a morphism of algebras and the diagram 
$$
\xymatrix@C35pt@R25pt{
\K(P(K), S_K) & \K[K]/(x_i^2) \ar[l]^-{\cong}_-{\alpha_K} \\
\K(P(L), S_L)\ar[u]^{P(\varphi)^*} & \K[L]/(y_j^2)\ar[u]_{\varphi^*}  \ar[l]_-{\cong}^-{\alpha_L} 
}
$$ 
is commutative.  
\end{prop}

\begin{proof}
By the definition of $\alpha_L$, it follows that 
$\alpha_L(x_{i_1}\dots x_{i_n})=\widetilde{\{{i_1}, ..., {i_n}\}}$.  Moreover, we have 
\begin{eqnarray*}
P(\varphi)^*(\widetilde{\{{i_1}, ..., {i_n}\}}) &=& 
\sum_{\varphi(\{{j_1}, ...., {j_s}\}) = \{ {i_1}, ..., {i_n}\}} \widetilde{\{{j_1}, ...., {j_s}\}}\\
&=&\sum_{\varphi(\{{j_1}, ...., {j_n}\}) = \{ {i_1}, ..., {i_n}\}} \widetilde{\{{j_1}\}}\cdots \widetilde{\{{j_n}\}}\\
&=& \sum_{(\varphi({l_1}), ..., \varphi({l_n}))= ({i_1}, ..., {i_n})} \widetilde{\{{l_1}\}}\cdots \widetilde{\{{l_n}\}}\\
&=& \sum_{{l_1} \in \varphi^{-1}({i_1}), ..., {l_n} \in \varphi^{-1}({i_n})}\widetilde{\{{l_1}\}}\cdots \widetilde{\{{l_n}\}}\\
&=& \alpha_K(\sum_{{l_1} \in \varphi^{-1}({i_1})}y_{l_1}\cdots \sum_{{l_n} \in \varphi^{-1}({i_n})}y_{l_n})\\
&=& (\alpha_K\circ \varphi^*)(x_{i_1}\dots x_{i_n}). 
\end{eqnarray*}
Since $\varphi$ is a non-degenerate simplicial map, it follows that in the equations above, $s=n$ and that 
$\varphi(\{{j_1}, ..., {j_n}\}) = \{ {i_1}, ..., {i_n}\}$ if and only if $(\varphi({l_1}), ..., \varphi({l_n}))= ({i_1}, ..., {i_n})$ 
with $\{l_1, ..., l_n \} = \{j_1, ..., j_n\}$. Then the diagram is commutative and 
hence $P(\varphi)^*$ is a morphism of algebras. 
\end{proof}

\begin{rem} Let $\mathsf{SimpComp}_{\text{non-deg}}$ be the wide category of simplicial complexes with 
non-degenerate simplicial maps as morphisms.  Then we have a functor 
$$P : \mathsf{SimpComp}_{\text{non-deg}} \to q\mathsf{ASmd}$$ which is faithful but not full. 
In fact, let $K$ be the simplicial set with a single vertex $1$ 
and $L$ the standard simplex with two points $1$ and $2$.  
Then the map $F: (P(K), S_K) \to (P(L), S_L)$  defined by 
$F(f : \phi \to \{1\})) = (\phi \to \{1, 2\})$ is a morphism of quasi-schemoids but not in 
the image of the functor $P$. 
\end{rem}

\begin{rem}
Let $(X, {\mathcal O})$ be a topological space and $\text{Open}_X$  the category of whose objects are open sets of $X$ and 
whose morphisms are inclusions. Then Lemma \ref{lem:construction} enables us to construct a schemoid $(\text{Open}_X, S_X)$. 
Moreover, we have a model category structure of $Ch(\T^{(\text{Open}_X, S_X)})$ 
and obtain a derived category of the form $\text{D}(\T^{(\text{Open}_X, S_X)})$ if $X$ is finite; see Remark \ref{rem:height_functor}. 
\end{rem}

A finite poset $X$ is considered a finite $T_0$-space. Then one might expect that the model category structure of 
$Ch(\T^{(\text{Open}_X, S_X)})$ mentioned above is of great use in the study of finite posets. The topic will be addressed 
in forthcoming work.   

Let $f : X \to Y$ be a continuous map. Then we have a functor $f^* : \text{Open}_Y \to \text{Open}_X$ defined by $f^*(U) = f^{-1}(U)$. 

\begin{lem} The functor $f^*$ mentioned above gives rise to a morphism $f^* : (\text{\em Open}_Y, S_Y) \to (\text{\em Open}_X, S_X)$ of schemoids. 
\end{lem}

\begin{proof} Let $i : U \to V$ be in $\widetilde{\sigma}$, where $\widetilde{\sigma} \in S_Y$. 
We have  $V = U\sqcup \sigma$ and hence $f^{-1}(V) = f^{-1}(U) \sqcup f^{-1}(\sigma)$. This implies that 
$f^*(i) \in \widetilde{f^{-1}(\sigma)}$. 
\end{proof}

\begin{rem}
The functor category $\T^{(\text{Open}_X, S_X)}$ is an abelian subcategory of the category $\T^{\text{Open}_X}$ of presheaves of $\K$-modules over $X$. We see that there is no essential intersection between $\T^{(\text{Open}_X, S_X)}$ and sheaves. In fact, 
suppose that $F$ is a sheaf in $\T^{(\text{Open}_X, S_X)}$. Then $F(\varnothing) = 0$. For any open set $U$, $id_U : U \to U$ and 
$\varnothing \to \varnothing$ are contained in a common $\sigma \in S_X$. This enables us to conclude that 
$F(id_\varnothing) = F(id_U) : F(U) = F(\varnothing) \to F(\varnothing) = F(U)$. We have $F(U) =0$. Therefore, $F$ is nothing but the constant sheaf $\underline{0}$. 
\end{rem}

\end{document}